\documentclass[1p]{elsarticle}

\makeatletter
\def\ps@pprintTitle{%
  \let\@oddhead\@empty
  \let\@evenhead\@empty
  \def\@oddfoot{}%
  \let\@evenfoot\@oddfoot}
\makeatother

\usepackage{amsmath,amsthm,amssymb,dsfont,epsfig,booktabs,xfrac}
\usepackage{hyperref}

\usepackage{caption,sidecap,cancel}
\usepackage{wrapfig}

\makeatletter
\newcommand{\opnorm}{\@ifstar\@opnorms\@opnorm}
\newcommand{\@opnorms}[1]{%
  \left|\mkern-1.5mu\left|\mkern-1.5mu\left|
   #1
  \right|\mkern-1.5mu\right|\mkern-1.5mu\right|
}
\newcommand{\@opnorm}[2][]{%
  \mathopen{#1|\mkern-1.5mu#1|\mkern-1.5mu#1|}
  #2
  \mathclose{#1|\mkern-1.5mu#1|\mkern-1.5mu#1|}
}
\makeatother

\usepackage[normalem]{ulem}

\usepackage[usenames,dvipsnames]{xcolor}
\usepackage{tikz-cd}
\usetikzlibrary{calc,arrows,positioning,pgfplots.groupplots,matrix}
\usepackage{pgfplots}
\pgfplotsset{compat=1.14}

\usepackage{accents}

\newtheorem{lemma}{Lemma}

\newtheorem{remark}[lemma]{Remark}

\newcommand{\blue}[1]{{~}}

\newcommand{\N}{ \mathcal{N} }

\newcommand{\vt}{\mathbf{v}}
\newcommand{\Vt}{W}
\newcommand{\ft}{\mathbf{f}}

\newcommand{\ut}{\boldsymbol{\omega}}
\newcommand{\xt}{\mathbf{x}}
\newcommand{\zt}{\mathbf{z}}
\newcommand{\nt}{\vec{n}}

\renewcommand{\div}{\operatorname{div}}

\newcommand{\V}{{\cal V}}
\renewcommand{\L}{{\cal L}}

\newcommand{\dd}{\,\mathrm{d}}

\begin{document}

\begin{frontmatter}
  
  \title{A priori and a posteriori error estimates for the Deep Ritz method 
    applied to the Laplace and Stokes problem}
  
  
  \author[1]{P. Minakowski}
  \ead{piotr.minakowski@ovgu.de}
  
  \author[1]{T. Richter}
  \ead{thomas.richter@ovgu.de}
  \ead[url]{numerics.ovgu.de}
  
  \address[1]{Otto-von-Guericke Universit\"at  Magdeburg,  
    Universit\"atsplatz 2, 39106, Magdeburg, Germany}
  
  \begin{abstract}
    We analyze neural network solutions to partial differential equations
    obtained with \emph{Physics Informed Neural Networks}.
    In particular, we apply tools of classical finite element error analysis to
    obtain conclusions about the error of the \emph{Deep Ritz} method applied 
    to 
    the Laplace and the Stokes equations. 
    Further, we develop an a posteriori error estimator for neural network 
    approximations of
    partial differential equations. The proposed approach is based on
    the \emph{dual weighted residual estimator}.  It is destined to serve as
    a stopping criterion that guarantees the accuracy of  
    the solution independently of the design of the neural network training.
    The result is equipped with computational examples for Laplace and Stokes 
    problems. 
  \end{abstract}
  
  
\end{frontmatter}

\section{Introduction}\label{sec:intro}

In recent years, the emerging field of (deep) neural networks has
reached the numerical approximation of partial differential
equations (PDE). Several approaches have been proposed that aim at
representing the solution to the PDE by a deep neural
network. Many of these approaches have demonstrated that they can
provide efficient approximations in certain situations.
Here we pursue two goals. On the one hand, we try to show
commonalities in the analysis of finite element methods and the
analysis of solutions represented by neural networks. On the other
hand, we develop an a posteriori error estimator which allows testing
a once trained network for its accuracy and which can furthermore be
used as a termination criterion during the training process.
Even though the resulting error estimator provides very
high accuracy in practical applications, it is not robust in the sense
of a guaranteed upper bound.

\paragraph{Learning solutions of partial differential equations}
We consider a neural network function  $v_\N$ as a (differentiable) function $v_\N:\Omega\to\mathds{R}^c$, where $\Omega\subset\mathds{R}^d$ is the
computational domain of dimension $d\in\mathds{N}$ and
$c\in\mathds{N}$ is the size of the differential system. By $\N$ we denote the architecture of a neural network, $v_\N$ is then a specific realization within this architecture.
Several different approaches have been presented that train the
network by integrating the differential equation into a loss
function.
Here we focus on the \emph{Deep Ritz} method by E and Yu~\cite{EYu2017}
which aims at minimizing the energy functional and can be applied
to symmetric problems. 
For the Laplace equation, $-\Delta u=f$ in $\Omega$ with $u=g$ on $\partial\Omega$, this means to minimize
$$
E(u_\N)\le 
E(v_\N) = \frac{1}{2} \int_\Omega |\nabla v_\N(x)|^2\dd x - \int_\Omega v_\N(x)\cdot f(x)\dd x
+ \lambda \int_{\partial\Omega} |v_\N(x)-g(x)|^2\dd x,
$$
with a parameter $\lambda>0$, where the integrals are approximated by 
Monte-Carlo integration based on randomly chosen quadrature points within 
the domain $\Omega$ and on the boundary $\partial\Omega$. The optimal network 
solution $u_\N$ is then identified by minimizing the approximated energy 
functional $E(u_\N)\le E(v_\N)$. See Section~\ref{sec:drlap} 
and~\cite{Samaniego2020} for an 
overview and further examples on the energy based approach. 

Another approach, denoted as \emph{DeepXDE} (Lu, Meng, Mao and
Karniadakis~\cite{Lu2021}), \emph{Unified Deep Artificial
Network} \cite{Berg2018} or \emph{Deep Galerkin
Method (DGM)} \cite{Sirignano2018}, see also
\cite{Anitescu2019}, minimizes the strong  
residual of the equation, either as collocation method in 
randomly picked $N^{in}$ points within the domain and $N^{bnd}$ on the boundary, formulated once more for the Laplace problem,
\[
E(u_\N)\le 
E(v_\N) = \frac{1}{N^{in}}\sum_{k=1}^{N^{in}} |- \Delta v_\N(x_k)-f(x_k)|^2 +  \frac{\lambda}{N^{bnd}}\sum_{j=1}^{N^{bnd}} |v_\N(x_j)-g(x_j)|^2. 
\]
This corresponds to a Monte-Carlo integration of the (strong) residual
\[
E(v_\N) = \|- \Delta v_\N -f\|^2_{L^2(\Omega)} + \lambda
\|g-v_\N\|_{L^2(\partial\Omega)}, 
\]
which could be considered as a natural extension of \emph{Deep Ritz}
to non-symmetric and nonlinear problems.

Finally, a third variant, variational physics-informed neural network 
\emph{VPINN} (Kharazmi, Zhang and Karniadakis \cite{KharazmiZhangKarniadakis2019})
is based on the variational formulation 
\[
E(u_\N)\le E(v_\N)= \left( \int_\Omega \nabla v_\N(x)\cdot \nabla\phi_k(x)\dd x - \int_\Omega f(x) \phi_k(x)\dd x\right)^2
\]
and \emph{training data} is generated by choosing specific test
functions $\phi_k$.

The stopping criteria of the training process have not been studied in
detail. Commonly a fixed number of epochs is performed, see
e.g. \cite{EYu2017} or training runs until the mean residual reaches 
a specified threshold,~\cite{Lu2021}. According to the authors
knowledge, it is the first time that an a posteriori error estimator is
utilized as stopping criterion.

\paragraph{Approximation properties and convergence}
The common rationale for the three different approaches discussed above is the 
excellent approximation property of neural networks, in particular,
Pinkus~\cite{Pinkus1999} proved the capability of deep neural networks 
to uniformly 
and simultaneously approximate differential functions and their derivatives. In 
the context of PDEs, Ghring and coworkers~\cite{Ghring2019} 
showed approximation results in Sobolev spaces and also gave convergence rates 
in the number of layers, neurons and weights. In 
particular for high dimensional differential equations, deep neural network 
based approaches promise to be 
superior~\cite{Lu2021,Sirignano2018}. On the other hand, it must be noted that 
the previous approaches, applied to common, 
low-dimensional ($d=1,2,3$) problems, 
cannot compete with established methods in terms of efficiency. While 
algorithms of $O(N)$ complexity exist for finite element or finite 
difference approximations of elliptic problems, the training of the deep neural 
network is a far more challenging computational task.

\paragraph{Learning operators}
All approaches mentioned so far have in common that a neural network
represents the solution of a specific differential equation
problem. If the problem, e.g. the right-hand side, a parameter or the
domain, is changed, a new network must be trained.
This is also the case for classical simulation methods. Here however the 
solution is obtained by solving a linear or nonlinear
  system of equations which can be accomplished with optimal
  efficiency if, for instance,  multigrid methods are
  available. In contrast, neural network based approaches will call
  for a retraining of the network, which corresponds to solving an
  ill-structured optimization problem.
The picture changes if
e.g. parameter-dependent learning is used or if very high-dimensional
problems are investigated~\cite{EYu2017}. 
\emph{DeepONet}~\cite{deeponet} extends the above mentioned ideas and directly 
aims at learning the complete solution operator. This allows
reusing the once trained network for solving multiple problems.

Deep learning techniques can be applied in the context of numerical simulation, 
as an extension of existing CFD codes to increase their efficiency. One can 
generalize existing numerical methods as artificial neural networks with a 
set of trainable parameters. In \cite{Mishra2018} the authors 
recast finite volume schemes as neural networks and train the underlying parameters 
to improve accuracy on coarse grids, for the solution of time-dependent ODEs and PDEs.
This approach was extended to finite element methods
in~\cite{Brevis2020}.
The \emph{Deep Neural Network Multigrid Method
  (DNN-MG)}~\cite{HartmannLessigMargenbergRichter2020,Margenberg_2021}, 
uses deep neural networks to locally enrich classical finite element
multigrid solutions on coarse meshes with additional fine mesh fluctuations.

One contribution that is similar in terms of the techniques is
the use of neural networks to represent dual problems in the context
of the DWR method~\cite{RothSchroederWick2022}.

\paragraph{Outline of the article}
In this contribution, which is based on the early
preprint~\cite{minakowski2021error}, we 
tackle the question of reliability and error control of deep neural  
network approaches. 

After a brief
  introduction in Section~\ref{sec:intro}, we focus on the \emph{Deep Ritz} 
  method, and investigate to what extent classical finite element analysis 
  carries over to neural network approximations by analysing different 
  contributions to the approximation error. 
  Section~\ref{fesetup} is devoted to the numerical analysis of network 
approximations to the Laplace and the Stokes problem. We present a priori error 
  estimates that are complete except for the error stemming from the inexact 
  solution of the optimization problem.

Next, in Section~\ref{sec:error}, we derive a posteriori error bounds for neural
network solutions. This error estimator extends to different PINNs and can 
be used to rate the quality of trained networks. 
A posteriori error estimation is approached within 
the concept of the \emph{dual weighted residual estimator} (DWR) that has been 
introduced by Becker and Rannacher~\cite{BeckerRannacher2001}. What we derive 
is not a rigorous bound, but an efficient computational tool that can be 
used to validate neural network solutions and  serve as an estimate in stopping 
criteria during the network training. For the sake of simplicity, the 
estimator is developed for the \emph{Deep Ritz} method. However, it
  directly extends to different neural network representations of the
  solution.

Later we briefly present network architecture in Section~\ref{sec:setup}.
Section~\ref{sec:num} demonstrates the accuracy of the 
estimator for different applications and shows how the 
estimator can be integrated as a stopping criterion in training. 
After presenting different numerical examples we conclude in 
Section~\ref{sec:conc}. 


\section{Finite element and neural network approximations}\label{fesetup}

To keep the notation simple we focus on the Laplace
problem. Let $\Omega\subset\mathds{R}^d$ be a $d$-dimensional
domain. We find the weak solution $u$ $\in \V:=H^1_0(\Omega)$ to
\begin{equation}\label{lap:strong}
  -\Delta u = f\text{ in }\Omega,\quad u=0\text{ on }\partial\Omega,
\end{equation}
where $f\in L^2(\Omega)$ is the right hand side. By $\V=H^1_0(\Omega)$ we
denote the space of $L^2$-functions with the first weak derivative in
$L^2(\Omega)$ with a vanishing trace on $\partial \Omega$. The solution $u\in 
\V$ is 
characterized by the variational problem
\begin{equation}\label{lap:var}
  (\nabla u,\nabla v) = (f,v)\quad\forall v\in \V,
\end{equation}
where we denote by $(\cdot,\cdot)$ the $L^2$-inner product on
$\Omega$. Further, the solution is also equivalently characterized as the
minimizer of the functional
\begin{equation}\label{lap:energy}
  E(u)\le E(v):=\frac{1}{2}\|\nabla v\|^2 - (f,v)\quad\forall v\in \V,
\end{equation}
where $\|\cdot\|$ is the $L^2$-norm on $\Omega$. 

\subsection{Finite element approximation}\label{setup:fe}

Now, let $\Omega_h$ be a triangulation of $\Omega$ into open
triangular or quadrilateral (in 2d) elements satisfying usual
regularity requirements on the structure and the form of the elements. For an 
element $T\in\Omega_h$ we denote by
$h_T=\operatorname{diam}(T)$ the element size and by
$h = \max_{T\in\Omega_h} h_T$
the maximum mesh size of the discretization which serves as a parameter
for measuring the fineness. 

By $V_h\subset\V$ we denote the finite dimensional (finite
element) subspace of $H^1_0(\Omega)$. Then, let $u_h\in V_h$ be the
approximation to $u\in\V$ given by
\begin{equation}\label{lap:fe}
  (\nabla u_h,\nabla v_h)=(f,v_h)\quad\forall v_h\in V_h.
\end{equation}
The finite element error $u-u_h$ is bounded by the
interpolation error, yielding the standard estimate
\begin{equation}\label{lap:error}
  \|\nabla (u-u_h) \| \le c h^r \|f \|_{H^{r-1}(\Omega)},
\end{equation}
where $r$ is the polynomial degree of the finite element space and
using the notation $H^0(\Omega):=L^2(\Omega)$ in the case of linear
finite elements, $r=1$. Naturally, this estimate requires sufficient
regularity of the right hand side $f\in H^{r-1}(\Omega)$ and also of
the domain boundary, i.e. $\partial\Omega$ must be a convex polygonal
for $r=1$ or locally parametrizable by a $C^{r+1}$-function for $r\ge
1$. 

\subsection{Deep Ritz approximation of the Laplace problem}\label{sec:drlap}

In principle, the \emph{Deep Ritz} method as proposed by E and Yu~\cite{EYu2017} is 
based on minimizing the energy 
functional~\eqref{lap:energy} by representing the unknown solution
by a neural network $u_{\N,\ut}:\mathds{R}^d\to\mathds{R}$ instead of 
a finite element function. We denote by $\N$ the topology of the
neural network and by $\ut\in \mathds{R}^{\#\N}$ the parameters of the
network, where $\#\N$ is the total number of free
parameters. Finally, $u_{\N,\ut}$ is the function that is realized by 
this specific combination of network topology and parameter
choice. Mostly, we will simply use the notation $u_\N$ and skip the
indication of the parameter vector $\ut$ unless it is of relevance in
the given context.

The framework of 
the \emph{Deep Ritz} method requires differentiability of 
the network, i.e. differentiable activation functions. Fig.~\ref{fig:network} 
shows the layout of the deep neural network as chosen by E and 
Yu, but also a simpler feedforward network that can be used.

Having a certain network topology $\N$ in mind, the \emph{neural
network approximation space} $V_\N$ used in the \emph{Deep Ritz} method is
given by
\begin{equation}\label{network:space}
  V_\N:=\{u_{\N,\ut}:\mathds{R}^d\to \mathds{R}\,|\, \ut\in\mathds{R}^{\#\N}\}.
\end{equation}

As long as $\#\N$ is finite and when the activation functions are
differentiable it holds  
$V_\N\subset H^1(\Omega)$. However, for $u_\N\in V_\N$ it will not hold $u_\N=0$ on $\partial\Omega$ in the general case, hence $V_\N\not\subset
\V=H^1_0(\Omega)$\footnote{In~\cite{Berg2018}, the authors discussed a modified
setup of the neural network that indeed strongly satisfies the homogeneous
Dirichlet condition.}.
Further, it is important to note that $\V_\N$ is not a vector space. For 
$v_1,v_2\in \V_\N$ it will usually not hold that $v_1+v_2\in \V_\N$.
We consider the penalized energy functional, compare \cite{EYu2017},
\begin{equation}\label{lapvar:energy}
  E_\lambda(v):= \frac{1}{2}\|\nabla v\|^2 - (f,v) + 
  \frac{\lambda}{2}|v|^2_{\partial\Omega},
\end{equation}
where $\lambda\in\mathds{R}_+$ is a parameter and 
$|\cdot|_{\partial\Omega}$ is the $L^2$-norm on the boundary of the domain. The 
additional penalty term forces $v$ towards zero along the boundary. The 
minimizer of~\eqref{lapvar:energy} denoted by $u_\lambda\in H^1(\Omega)$ is 
characterized by the variational problem
\begin{equation}\label{lapvar:var}
  a_\lambda(u_\lambda,v) = (f,v)\quad\forall v\in H^1(\Omega),\quad
  a_\lambda(u,v):= (\nabla u,\nabla v) + \lambda \langle
  u,v\rangle_{\partial\Omega},
\end{equation}
where $\langle\cdot,\cdot\rangle_{\partial\Omega}$ denotes the $L^2$-inner product on the boundary $\partial\Omega$. 
The weak solution $u_\lambda\in H^1(\Omega)$ solves the Laplace
problem with a disturbed Robin boundary condition, i.e. 
\begin{equation}\label{lapvar:strong}
  -\Delta u_\lambda = f\text{ in }\Omega,\quad u_\lambda +\lambda^{-1}\partial_n u_\lambda = 0\text{ on }\partial\Omega.
\end{equation}
The penalized energy functional hence introduces an additional
modeling error term $\|u-u_\lambda\|$ that will depend on the parameter
$\lambda$ and that will converge to zero for $\lambda\to
\infty$.

Training of the neural network minimizes the modified energy functional~(\ref{lapvar:energy}) 
expressed in terms of Monte Carlo integration. To be precise:
$N^{in}\in\mathds{N}$ inner quadrature points $\xt^{in}=\{x^{in}_1,\dots,x^{in}_{N^{in}}\}\subset \Omega$ and $N^{bnd}\in\mathds{N}$ boundary quadrature points
$\xt^{bnd}=\{x^{bnd}_1,\dots,x^{bnd}_{N^{bnd}}\}\subset\partial\Omega$
are chosen, either randomly or based on a mesh of the domain. The loss function is given by
\begin{multline}\label{lap:lossfunction}
  E_{\lambda,\text{mc}}(u_\N; \mathbf{x}^{in},\mathbf{x}^{bnd}):=\\\frac{|\Omega|}{N^{in}}\sum_{k=1}^{N^{in}}\Big(\frac{1}{2}|\nabla u_\N(x^{in}_k)|^2
  -f(x^{in}_k)\cdot u_\N(x^{in}_k)\Big)+\frac{|\partial\Omega|}{N^{bnd}}
  \sum_{j=1}^{N^{bnd}}\frac{\lambda}{2} |u_\N(x^{bnd}_j)|^2.
\end{multline}
Minimizing~\eqref{lap:lossfunction} identifies the weights $\ut\in
\mathds{R}^{\#\N}$.

\paragraph{Error analysis for the Laplace equation}
  For the following, we denote by $u\in H^1_0(\Omega)$ the exact solution to 
  the Laplace problem and by $u_\N\in V_\N$ the \emph{Deep Ritz} solution that 
  is 
  obtained with a numerical optimization routine and which is based on 
  Monte-Carlo integration of the energy functional. The error of the \emph{Deep 
  Ritz} approximation $(u-u_\N)$ is composed of a multitude of different 
  influences: first, as stated above, the energy functional is based on a 
  perturbed problem, and we denote by $u_\lambda\in H^1(\Omega)$ the exact 
  solution to this perturbed problem. By $(u-u_\lambda)$ we denote the 
  \emph{modeling error}.
  By $ u_{\N,ex}\in V_\N$ we denote the minimum of the energy functional 
  \eqref{lapvar:energy} in the set of neural network functions $V_\N$. 
  The error $(u_\lambda-u_{\N,ex})$ is an \emph{approximation error}. Next, $u_{\N,\text{mc}}\in  V_\N$ is minimum of the  Monte-Carlo approximated 
  energy functional \eqref{lap:lossfunction}. This introduces the  
  \emph{generalization error}
  $(u_{\N,ex}-u_{\N,\text{mc}})$.
  Finally, the \emph{optimization error} 
  $(u_{\N,\text{mc}}-u_\N)$ remains.
  Altogether, four distinct
  contributions can be identified
\[
\|u-u_{\N}\| \le
\|u-u_\lambda\| +
\|u_\lambda-u_{\N,ex}\| +
\|u_{\N,ex} - u_{\N,\text{mc}}\|+
\|u_{\N,\text{mc}}-u_{\N}\|.
\]
Modeling (1st) and network approximation (2nd) error of the Laplace 
equation have been studied in the literature, and also quantitative convergence 
results are available~\cite{MuellerZeinhofer2021,Liao2021}, these however do 
not consider   the numerical quadrature of the energy functional. Some first 
results are also  known for nonlinear problems~\cite{Dondl2021}.
Also, the generalization error (3rd) has been studied, usually from a 
stochastic point of view~\cite{LuLuWang2021}.
We start by estimating the 
model error that depends on the choice of $\lambda>0$, but that is not yet 
related to the discretization of the equation.

\begin{lemma}[Model error]\label{lemma:modelerror}
  Let $f\in L^2(\Omega)$, $\lambda\in\mathds{R}_+$ and
  $\Omega$ be such that the solutions $u\in H^1_0(\Omega)$ and
  $u_\lambda\in H^1(\Omega)$ to 
  \begin{equation}\label{lem:mod:1}
    (\nabla u,\nabla\phi) = (f,\phi),\qquad
    (\nabla u_\lambda,\nabla\phi_\lambda) + \lambda \langle
    u_\lambda,\phi_\lambda\rangle_{\partial\Omega} =
    (f,\phi_\lambda)
  \end{equation}
  for $\phi\in H^1_0(\Omega)$ and $\phi_\lambda\in H^1(\Omega)$ 
  satisfy $\|u\|_{H^2(\Omega)}+\|u_\lambda\|_{H^2(\Omega)}\le c_s \|f\|$.
  It holds
  \[
  \|\nabla (u-u_\lambda) \| \le \frac{c}{\lambda} \|f\|,
  \]
  where $c>0$ depends on the domain $\Omega$ only. 
\end{lemma}
\begin{proof}
  Let $z\in H^1_0(\Omega)$ be the solution to the adjoint problem
  \begin{equation}\label{n:adjoint}
    -\Delta z = \frac{\nabla (u-u_\lambda)}{\|\nabla
      (u-u_\lambda)\|}\text{ in }\Omega,\quad
    z=0\text{ on }\partial\Omega.
  \end{equation}
  Since the right hand side of this problem is in $L^2(\Omega)$, it
  holds $z\in H^2(\Omega)$ and $\|z\|_{H^2(\Omega)}\le c_s$ 
  (given that the domain's boundary is sufficiently smooth or convex
  polygonal).   
  Multiplication of~(\ref{n:adjoint}) with the error $u-u_\lambda$ and
  integration over the domain give the error identity
  \[
  \|\nabla (u-u_\lambda)\| = (\nabla z,\nabla (u-u_\lambda))
  - \langle \partial_n z,u-u_\lambda\rangle_{\partial\Omega}. 
  \]
  As $u=0$ and $z=0$ on $\partial\Omega$ this, together
  with~\eqref{lem:mod:1}  gives 
  \[
  \|\nabla (u-u_\lambda) \| = (\nabla (u-u_\lambda),\nabla
    z)+\lambda\langle u-u_\lambda,z\rangle_{\partial\Omega}+ \langle
    \partial_n z,u_\lambda\rangle_{\partial\Omega}
    =\langle \partial_n z,u_\lambda\rangle_{\partial\Omega}. 
  \]
  Finally, with~(\ref{lapvar:strong}) and using the trace inequality
  and the regularity of  adjoint and primal solution we obtain
  \[
  \|\nabla (u-u_\lambda) \|
  \le | \partial_n z|_{L^2(\partial\Omega)}
  | u_\lambda|_{L^2(\partial\Omega)}
  \le \frac{c}{\lambda} \|z\|_{H^2} \|u_\lambda\|_{H^2(\Omega)}
  \le \frac{c}{\lambda}\|f\|.
  \]
\end{proof}

To study the \emph{approximation error} $u_\lambda - u_{\N,ex}$, where  $u_\lambda\in H^1(\Omega)$ is the minimizer to $E_\lambda(\cdot)$ in the Hilbert space $H^1(\Omega)$ and $u_{\N,ex}\in V_\N$ is the minimizer to $E_\lambda(\cdot)$ in the neural network set, we use a generalized version of Cea's leamma taken from~\cite[Prop. 3.1]{MuellerZeinhofer2021}. It holds
\begin{equation}\label{cea}
  \|u_\lambda - v_\N\|_\lambda^2 \le
  2 \big( E_\lambda(v_\N)-\inf_{\tilde v\in V_\N} E_\lambda(\tilde v) \big)
  + \inf_{w_\N\in V_\N}\| u_\lambda - w_\N\|_\lambda^2\quad \forall v_\N\in V_\N,
\end{equation}
where the norm $\|\cdot\|_\lambda$ is defined via the bilinear form~\eqref{lapvar:var}
\[
\|u\|_\lambda := a_\lambda(u,u)^\frac{1}{2}=\Big( \|\nabla u\|^2 + \lambda 
\|u\|_{\partial\Omega}^2\Big)^\frac{1}{2}.
\]
If we choose $v_\N=u_{\N,ex}\in V_\N$, the exact minimum to $E_\lambda(\cdot)$ in the neural network set, it holds
\[
\|u_\lambda - u_{\N,ex}\|_\lambda \le
\inf_{w_\N\in V_\N}
\| u_\lambda - w_\N\|_\lambda\quad \forall w_\N\in V_\N,
\]
and the error $u_\lambda-u_{\N,ex}$ is indeed a neural network approximation 
error that has already been extensively studied in 
the literature~\cite{Barron1993a,Ghring2019,MuellerZeinhofer2021}. Quantitative convergence results 
$\|u_\lambda-u_{\N,ex}\|\to 0$ for increasing network sizes are well 
known, see 
for example~\cite[Theorem 4.1]{Ghring2019}, where the authors state that the 
bound 
\begin{equation}\label{approx1:1}
  \|\nabla (u_\lambda- u_{\N,ex})\| = {\cal O}(\epsilon)
\end{equation}
is obtainable, given $u_\lambda\in H^2(\Omega)$, with a neural network
consisting of $L={\cal O}(\log(\epsilon^{-2}))$ layers and
$N={\cal O}(\epsilon^{-2}\log(\epsilon^{-2}))$ weights and neurons. 
This approximately corresponds to\footnote{The exact relation is $\epsilon=\sqrt{W_0(N)}/\sqrt{N}$, where $W_0(x)$ is the Lambert \emph{W}-function, the inverse of $f(y)=ye^y$. It holds $W(N)\le\log(N)$.}
\begin{equation}\label{approx1:2}
  \|\nabla (u_\lambda- u_{\N,ex})\| = {\cal O}\big(
  \frac{\sqrt{\log(N)}}{\sqrt{N}}\big) 
\end{equation}
which, in terms of the number of unknowns $N$, is comparable to the
number of unknowns in linear finite element approximation.

In Lemma~\ref{lemma:gen:laplace}, we will give a unified estimate for this approximation error and the generalization error, i.e. we will estimate $u_\lambda - u_{\N,\text{mc}}$ at once.

The choice of the numerical quadrature points gives rise to the
\emph{generalization error} of the neural network representation. In the
context of \emph{Deep Ritz}, the generalization error is the error of
numerical quadrature, i.e. the error $u_{\N,ex}-u_{\N,\text{mc}}$ where $u_{\N,\text{mc}}\in V_\N$ is the minimum 
in the neural network set $V_\N$ based on Monte-Carlo 
integration of the energy functional.
To estimate this error, one has to quantify the stability
of the minimizer with respect to the quadrature of the energy
functional. In~\cite{Mishra2022} the authors base the analysis on stability
estimates of the underlying partial differential equations. The
authors of~\cite{HongSiegelXu2021} give bounds on the Rademacher
complexity of the energy functions and therefore limit the generalization
error.

We will analyze the generalization error based on the best approximation 
estimate~\eqref{cea} taken from~\cite[Prop. 3.1]{MuellerZeinhofer2021} and on 
estimating the quadrature error between $E_\lambda(\cdot)$ and 
$E_{\lambda,\text{mc}}(\cdot)$. 
First, we  cite a standard result on the Monte Carlo integration error,
taken from~\cite[Theorem 2.1]{Caflisch1998}
\begin{lemma}[Monte Carlo Quadrature]\label{lemma:montecarlo}
  Let $\Omega\subset\mathds{R}^d$ be a bounded domain. 
  For large $N$, let $x_1,\dots,x_N\in \Omega$ be a set of Monte
  Carlo quadrature nodes. For $f\in C(\Omega)$ it holds
  \[
  \Big|\int_\Omega f(x)\,\text{d}x - \frac{|\Omega|}{N}\sum_{i=1}^N
  f(x_i)
  \Big| = \|f\|_{L^2(\Omega)}{\cal O}(N^{-\frac{1}{2}}\nu),
  \]
  where $\nu$ is a standard normal random variable.  
\end{lemma}
In a more precise version, the integration error depends on the
variance of $f$  and not on the $L^2$-norm of $f$ itself, this
simplified version however is sufficient for our purposes.

\begin{lemma}[Generalization and approximation error of the Deep Ritz method for the Laplace problem]
  \label{lemma:gen:laplace}
  Let $\Omega\subset\mathds{R}^d$ be a bounded domain, $f\in L^2(\Omega)$ and $V_\N\subset C^1(\Omega)\cap C(\bar\Omega)$ be a neural network set. Let $u_\lambda\in H^1(\Omega)$ be the solution to~\eqref{lapvar:var} and $u_{\N,\text{mc}}\in V_\N$ be the neural network minimizer to $E_{\lambda,\text{mc}}(\cdot)$.
  Further, let the network satisfy
  \begin{equation}\label{approx:mc}
    \inf_{v_\N\in V_\N}
    \|u_\lambda - v_\N\|_\lambda \le \epsilon_\N
  \end{equation}
  for a tolerance $\epsilon_\N>0$. Then, for $N:=\min\{N^{in},N^{bnd}\}$ it holds 
  \[
  \|\nabla (u_\lambda - u_{\N,\text{mc}})\|^2+
  \lambda\|\nabla (u_\lambda - u_{\N,\text{mc}})\|_{\partial\Omega}^2
  \le
  C \big( \epsilon_N^2 + {\cal O}(N^{-\frac{1}{2}}\nu) \big)
  \]
  where $\nu$ is a standard normal random variable and where $C>0$ depends on 
  the domain $\Omega$, $\|f\|_{L^2}$, $\|u_{\N,\text{mc}}\|_{C^1(\Omega)\cap 
  C(\bar\Omega)}$ and on $\|u_{\N,ex}\|_{C^1(\Omega)\cap C(\bar\Omega)}$.
\end{lemma}
\begin{proof}
  With the best approximation estimate~\eqref{cea} we get
  \begin{equation}\label{gen:1}
    \|u_\lambda - u_{\N,\text{mc}}\|_\lambda^2 \le
    2 \big( E_\lambda(u_{\N,\text{mc}})
    -\inf_{\tilde v_\N\in V_\N} E_\lambda(\tilde v_\N) \big)
    + \inf_{w_\N\in V_\N}\| u_\lambda - w_\N\|_\lambda^2.
  \end{equation}
  Here, $u_{\N,\text{mc}}$ is not the minimizer of $E_\lambda(\cdot)$ in $V_\N$, which we denote by $u_{\N,ex}\in V_\N$,  but the minimizer of $E_{\lambda,\text{mc}}(\cdot)$. Hence, we extend the first term as
  \begin{multline}\label{gen:2}
    E_\lambda(u_{\N,\text{mc}})-\inf_{\tilde v_\N\in V_\N} E_\lambda(\tilde v_\N)
    =:E_\lambda(u_{\N,\text{mc}})- E_\lambda(u_{\N,\text{ex}})\\
    =
    \big(E_\lambda(u_{\N,\text{mc}})-
    E_{\lambda,\text{mc}}(u_{\N,\text{mc}})\big)
    +\big(E_{\lambda,\text{mc}}(u_{\N,\text{mc}})
    -E_{\lambda,\text{mc}}(u_{\N,ex})\big)\\
    +\big(E_{\lambda,\text{mc}}(u_{\N,ex})-  E_\lambda(u_{\N,ex})\big).
  \end{multline}
  The first and the third terms are quadrature errors, and they can be 
  estimated with Lemma~\ref{lemma:montecarlo}, by using $V_\N\subset 
  C^1(\Omega)\cap 
  C(\bar\Omega)$ as well as Young's inequality
  \begin{multline*}
    \big|E_\lambda(u_{\N,\text{mc}})-
    E_{\N,\text{mc}}(u_{\N,\text{mc}})\big| \\
    =
    \Big(
    \|\nabla u_{\N,\text{mc}}\|^2_{L^4(\Omega)}
    + \|f\|_{L^2(\Omega)} \|u_{\N,\text{mc}}\|_{L^\infty(\Omega)}
    +\lambda \|u_{\N,\text{mc}}\|_{L^4(\partial\Omega)}^2
    \Big){\cal O}(N^{-\frac{1}{2}}\nu)\\
    \le
    C \Big( \|u_{\N,\text{mc}}\|^2_{W^{1,\infty}(\Omega)}
    + \lambda \|u_{\N,\text{mc}}\|^2_{L^\infty(\partial\Omega)}
    + \|f\|^2_{L^2(\Omega)}
    \Big) {\cal O}(N^{-\frac{1}{2}}\nu).
  \end{multline*}
  The same argument can be applied to estimate the last term in~\eqref{gen:2}. The second term of~\eqref{gen:2} is negative as
  \[
  E_{\lambda,\text{mc}}(u_{\N,\text{mc}}) = \inf_{v_\N\in V_\N}
  E_{\lambda,\text{mc}}(v_{\N}) \le E_{\lambda,\text{mc}}(u_{\N,ex})
  \]
  and therefore it can be neglected. The last term in~\eqref{gen:1} is the approximation error and given by~\ref{approx:mc}.
\end{proof}

Finally, the optimization error $(u_{\N,\text{mc}}-u_\N)$ remains for which there 
is no a priori error bound.

\subsection{Deep Ritz approximation of the Stokes equations}\label{intro:stokes}

As a second example we consider the Stokes equation on a two
dimensional domain $\Omega\subset\mathds{R}^2$, i.e. we find the velocity
$\vt\in \V^2_0:=H^1_0(\Omega)\times H^1_0(\Omega)$ and the pressure $p\in \L:=
L^2(\Omega)\setminus\mathds{R}$ such that
\begin{equation}\label{stokes:strong}
  \div\,\vt=0,\quad -\Delta \vt + \nabla p = \ft\text{ in }\Omega,\quad
  \vt=0\text{ on }\partial\Omega, 
\end{equation}
where we denote by $\ft\in \L^2:=L^2(\Omega)\times L^2(\Omega)$ the right hand
side. Considering a discrete pair of subspaces $V_h\times L_h\subset
\V\times \L$, the finite element solution is defined by
\begin{equation}\label{stokes:fe}
  (\div\,\vt_h,\xi_h)+(\nabla \vt_h,\nabla\phi_h)-(p_h,\nabla\cdot
  \phi_h)
  =(\ft_h,\phi_h)\quad\forall (\phi_h,\xi_h)\in V_h\times L_h. 
\end{equation}
Assuming inf-sup stability of the discrete finite element pair, the
solution exists uniquely, and standard best approximation results are satisfied,
e.g. for the $P^2-P^1$ Taylor-Hood element it holds
\begin{equation}\label{stokes:error:th}
  \|\nabla (\vt-\vt_h)\| + \|p-p_h\| \le c h^2 \|\ft\|_{H^1(\Omega)},
\end{equation}
or, for equal-order linear finite elements for the pressure and the velocity,
the solution to the stabilized formulation
\begin{equation}\label{stokes:festab}
  (\div\,\vt_h,\xi_h)+(\nabla \vt_h,\nabla\phi_h)-(p_h,\nabla\cdot
  \phi_h) + h^2 (\nabla p_h,\nabla\xi_h)
  =(\ft_h,\phi_h)\quad\forall (\phi_h,\xi_h)\in V_h\times L_h
\end{equation}
satisfies the estimate
\begin{equation}\label{stokes:error:eo}
  \|\nabla (\vt-\vt_h)\| + \|p-p_h\| \le c h \|\ft\|.
\end{equation}
We refer to the literature, e.g. the monograph of John~\cite{John2016} for
these and further aspects on the finite element approximations to the
Stokes equations. 

Having a saddle-point structure the Stokes system is not directly
associated to an energy form. Instead we realize the \emph{Deep Ritz} method
by introducing a penalty term to enforce the divergence free condition,
i.e.
\begin{equation}\label{energy:stokes}
  E_{\lambda,\alpha}(\vt):= \frac{1}{2} \|\nabla \vt\|^2 - (\ft,\vt) +
  \frac{\alpha}{2} \|\div\,\vt\|^2 + \frac{\lambda}{2}|\vt|^2_{\partial\Omega},
\end{equation}
where $\alpha,\lambda>0$ are two parameters controlling the balance
between minimizing the energy and satisfying the divergence constraint
and the boundary values. The solution is characterized by the
variational problem $\vt_{\lambda,\alpha}\in \V^2:=H^1(\Omega)\times H^1(\Omega)$
\begin{equation}\label{energy:stokes:variational}
  \begin{aligned}
    A_{\lambda,\alpha}(\vt_{\lambda,\alpha},\phi)&=F(\phi)\quad\forall \phi\in H^1(\Omega),\\
    A_{\lambda,\alpha}(\vt,\phi)&:=
    (\nabla\vt,\nabla\phi) + \alpha
    (\div\,\vt,\div\,\phi)+\lambda \langle
    \vt,\phi\rangle_{\partial\Omega},\\
    F(\phi) &:= (\ft,\phi). 
  \end{aligned}
\end{equation}
This variational problem corresponds to the following
classical formulation which also reveals a disturbed boundary
condition
\begin{equation}\label{stokes:boundary}
  -\Delta \vt_{\lambda,\alpha} - \alpha
  \nabla\div\,\vt_{\lambda,\alpha} = \ft\text{ in }\Omega,\quad
  \lambda\vt_{\lambda,\alpha} + \alpha \nt\div\,\vt_{\lambda,\alpha} - 
  \partial_n\vt_{\lambda,\alpha} = 0 \text{ on 
  }\partial\Omega. 
\end{equation}
Hereby and similar to Lemma~\ref{lemma:modelerror} we get
\begin{lemma}[Stokes model error]\label{lemma:stokes}
  Let $\ft\in \L^2$, $\lambda,\alpha\in\mathds{R}$ with
  $\lambda,\alpha>0$ and  $\Omega$ be such that the solutions
  $(\vt,p)\in \V^2_0\times \L$ and $\vt_{\lambda,\alpha}\in \V^2$
  to~(\ref{stokes:strong}) and~(\ref{energy:stokes:variational}),
  respectively,  satisfy
  $\|\vt\|_{H^2(\Omega)}+\|p\|_{H^1(\Omega)}\le c_s \|\ft\|$ and 
  $\|\vt_{\lambda,\alpha}\|_{H^2(\Omega)}\le c_s \|\ft\|$. 
  It holds
  \[
  \|\nabla (\vt-\vt_{\lambda,\alpha}) \| \le
  \frac{c}{\min\{\sqrt{\lambda},\sqrt{\alpha}\}}
  \|\ft\|, 
  \]
  where $c>0$ depends on the domain $\Omega$ only. 
\end{lemma}
\begin{proof}
  Due to its similarity to Lemma~\ref{lemma:modelerror} we just
  give a sketch of the proof. Considering the adjoint $(\zt,p)\in \V^2\times \L$, which is the solution to
  \[
  \frac{\nabla (\vt-\vt_{\lambda,\alpha})}{\|\nabla
    (\vt-\vt_{\lambda,\alpha})\|}
  =-\Delta \zt - \nabla q,\quad \div\,\zt=0,
  \]
  we obtain the error estimate
  \begin{multline}
    \|\nabla (\vt-\vt_{\lambda,\alpha})\| = \Big|
    (\nabla \zt,\nabla (\vt-\vt_{\lambda,\alpha})) +
    (q,\div\,(\vt-\vt_{\lambda,\alpha}))
    +\langle \partial_n
    \zt+q\nt,\vt_{\lambda,\alpha}\rangle_{\partial\Omega}\\
    \underbrace{-\alpha
      (\div\,\vt_{\lambda,\alpha},\div\,\zt)-\lambda\langle\vt_{\lambda,\alpha},\zt\rangle_{\partial\Omega}
      -(p,\div\,\zt)
    }_{=0}\Big|\\
    =\big|
    \langle \partial_n
    \zt+q\nt,\vt_{\lambda,\alpha}\rangle_{\partial\Omega}
    -(\div\,\vt_{\lambda,\alpha},q)\big|\\
    \le \big( \|\zt\|_{H^2(\Omega)} + \|q\|_{H^1(\Omega)}\big)\cdot \big( \|\div\,\vt_{\lambda,\alpha} \| + |\vt_{\lambda,\alpha}|_{\partial\Omega}\big). 
  \end{multline}
  On the other hand, diagonal testing of~(\ref{energy:stokes:variational}) gives
  \[
  \frac{1}{2}\|\nabla\vt_{\lambda,\alpha}\|^2 + \lambda |\vt_{\lambda,\alpha}|_{\partial\Omega}^2 + \alpha \|\div\,\vt_{\lambda,\alpha}\|^2 \le \frac{1}{2} \|\ft\|^2
  \]
  and hereby, we obtain the postulated result.
\end{proof}
For optimal scaling the two parameters $\alpha$ and $\lambda$
should be chosen similarly. This penalized energy minimization
formulation does not produce an approximation to the
pressure.

Having these first results at hand we can proceed as in the case of the 
Laplace problem and define $\vt_{\N,ex}\in W_\N=V_\N\times V_\N$ as the neural 
network solution based on exact integration. The numerical neural network 
solution is obtained by Monte-Carlo quadrature of the energy 
$E_{\lambda,\alpha}(\vt)$, see~\eqref{energy:stokes}, using $N^{in}$ interior 
and $N^{bnd}$ boundary points
\begin{multline}\label{loss:stokes}
  E_{\lambda,\alpha,\text{mc}}(\vt)
  :=
  \frac{|\Omega|}{N^{in}} \sum_{k=1}^{N^{in}}\Big\{
  \frac{1}{2}|\nabla\vt(x_k^{in})|^2
  +\frac{\alpha}{2} |\div\,\vt(x_k^{in})|^2
  - \ft(x_k^{in})\cdot \vt(x_k^{in})\Big\}\\
  +
  \frac{|\Omega|}{N^{bnd}} \sum_{j=1}^{N^{bnd}}
  \frac{\lambda}{2} |\vt(x_j^{bnd})|^2.
\end{multline}
The structure is comparable to the Laplace problem, see~\eqref{lap:lossfunction} just with the additional penalty term enforcing the divergence condition.

Similar to Lemma~\ref{lemma:gen:laplace} we then can estimate the generalization and approximation error of the Stokes problem.
\begin{lemma}[Generalization and approximation error of Deep Ritz (Stokes)]
    Let $\Omega\subset\mathds{R}^2$ be a bounded domain, $\ft\in L^2(\Omega)^2$ and $\Vt_\N\subset C^1(\Omega)^2\cap C(\bar\Omega)^2$ be a neural network set. Let $\vt_{\lambda,\alpha}\in H^1(\Omega)^2$ be the solution to~\eqref{energy:stokes} and~\eqref{energy:stokes:variational} and $\vt_{\N,\text{mc}}\in \Vt_\N$ be the neural network minimizer to $E_{\lambda,alpha,\text{mc}}(\cdot)$ given by~\eqref{loss:stokes}.
  Further, let the network be such that it holds
  \begin{equation}\label{approx:mc:stokes}
    \inf_{\vt_\N\in V_\N}
    \|\vt_{\lambda,\alpha} - \vt_\N\|_{\lambda,\alpha} \le \epsilon_\N,
  \end{equation}
  for a tolerance $\epsilon_\N>0$,
  where
  \[
  \|\vt\|_{\lambda,\alpha}^2:=
  \|\nabla \vt\|^2
  +\alpha \|\div \vt\|^2+
  \lambda\|\vt\|_{\partial\Omega}^2.
  \]
  Then, for $N:=\min\{N^{in},N^{bnd}\}$ it holds 
  \[
  \|\vt_{\lambda,\alpha} - \vt_{\N,\text{mc}}\|^2_{\lambda,\alpha}  \le
  C \big( \epsilon_N^2 + {\cal O}(N^{-\frac{1}{2}}\nu) \big)
  \]
  where $\nu$ is a standard normal random variable and where $C>0$ depends on the domain $\Omega$, of $\|\ft\|_{L^2}$, $\|\vt_{\N,\text{mc}}\|_{C^1(\Omega)^2\cap C(\bar\Omega)^2}$ and on $\|\vt_{\N,ex}\|_{C^1(\Omega)^2\cap C(\bar\Omega)^2}$.
\end{lemma}
The proof follows that of Lemma~\ref{lemma:gen:laplace} line by line, just taking into account the additional term $\alpha(\div\,\vt,\div\,\phi)$.

\section{A posteriori error estimation for neural network solutions}
\label{sec:error}

In the following we will derive an a posteriori error estimator for
estimating the complete error $(u-u_{\N})$ and $(\vt-\vt_{\N})$ for the Laplace 
and the Stokes problem, respectively, that includes all the
different error contributions discussed above: the \emph{model error}, the
\emph{approximation error}, the \emph{generalization error} and also the \emph{training
error}. This estimator is \emph{goal oriented}: instead of estimating the 
error in a norm $\|u-u_{\N}\|$, we estimate scalar quantities of interest 
$J:(u-u_\N)\mapsto \mathds{R}$. Examples of such error functionals are the 
point-wise error of the solution in a certain point $x_a\in\Omega$
\[
J_a(u-u_{\N}) = u(x_a)-u_\N(x_a),
\]
averages of the solution or boundary integrals on $\Gamma\subset\partial\Omega$
\[
J_\Omega(u-u_\N) = \int_\Omega u(x)-u_{\N}(x)\,\text{d}x,\quad
J_{\Gamma}(u-u_\N) = \int_\Gamma \partial_n
u(x)-\partial_nu_{\N}(x)\,\text{d}s.
\]

\subsection{The dual weighted residual method (DWR)}\label{sec:dwr}

We start by giving a concise description of the dual weighted
residual method, such as presented
in~\cite{BeckerRannacher1995,BeckerRannacher2001} for the most simple
case of the Laplace problem $-\Delta u=f$ with homogeneous Dirichlet
data $u=0$. In the general case, the dual weighted residual
  method can be applied to all problems given in a Galerkin
  formulation based on a bilinear or semilinear form. Applications
  include problems in fluid
  dynamics~\cite{BeckerRannacher1995,BeckerRannacher2001} and the
  approach has been extended to optimization and
  parameters identification~\cite{BeckerVexler2004}, to plasticity~\cite{RannacherSuttmeier1998} to coupled multiphysics
  problems~\cite{Richter2017}, among many other applications. The estimator is 
  further applicable to the estimation of time-stepping 
  errors~\cite{BesierRannacher2012,MeidnerRichter2015}.

As introduced above, $\V=H^1_0(\Omega)$ and $V_h\subset \V$ is a discrete subspace. 
Now, let $J:\V\to\mathds{R}$ be a linear functional (in the
  general case, the dual weighted residual estimator also handles
  nonlinear quantities of interest, see~\cite{BeckerRannacher2001}) and let $z\in \V$ be the solution to the adjoint problem
\begin{equation}\label{adjoint:strong}
  -\Delta z = J\text{ in }\Omega,\quad z=0\text{ on }\partial\Omega,
\end{equation}
which, in variational formulation, is given as
\begin{equation}\label{adjoint:var}
  z\in \V\quad (\nabla v,\nabla z) = J(v)\quad\forall v\in \V.
\end{equation}
This already gives the \emph{primal error identity}
\begin{equation}\label{ei:p}
  J(u-u_h) = (\nabla (u-u_h),\nabla z) = (f,z)- (\nabla u_h,\nabla z).
\end{equation}
By Galerkin orthogonality $(\nabla(u-u_h),\nabla v_h)=0$ for all $v_h\in
V_h$ the corresponding \emph{dual error identity} reads
\begin{equation}\label{ei:d}
  J(u-u_h) = (\nabla(u-u_h),\nabla(z-z_h)) = (\nabla u,\nabla (z-z_h)) = J(u)-(\nabla u,\nabla z_h),
\end{equation}
where $z_h\in V_h\subset \V$ is the discrete solution to the adjoint
problem.

Both simple error identities~(\ref{ei:p}) and~(\ref{ei:d}) cannot be
used in practice since the adjoint solution $z\in \V$ and the primal
solution $u\in \V$ are not known. Applying the DWR method calls for
an approximation of the primal or adjoint solutions in a subspace 
$V_{hh}\subset \V$ which is not a subspace of the discrete space,
e.g. $V_{hh}\not\subset V_h$ and to approximate the error  by 
\begin{equation}\label{ei:a}
  J(u-u_h) \approx \eta_h(u_h,z_{hh}):=(f,z_{hh}) - (\nabla u_h,\nabla z_{hh}).
\end{equation}
By $\eta_h(u_h,z_{hh})$, we denote the error estimator. It only depends 
on the primal and adjoint discrete solution and it is therefore computable.
Likewise, a computable estimator could be defined based on the adjoint error identity~\ref{ei:d}. While both formulations are equivalent for linear problems, a combination of primal and dual form is required in the general nonlinear case, see~\cite{BeckerRannacher2001}.

Various approaches are discussed in~\cite[Sec. 5]{BeckerRannacher2001}
or \cite[Sec. 3]{RichterWick2015} . In general, they are based on the higher-order
postprocessing of same-space approximations. This reconstruction of higher
order information is of an approximative type such that the DWR method
usually does not give a rigorous error bound but only a computational
measure to estimate the error in practical applications. This
approximation is usually highly accurate. 

Given that the exact error $J(u-u_h)$ is known, the accuracy of the
error estimator can be numerically validated by considering the
\emph{effectivity index}, which is the quotient of estimator value 
$\eta_h(u_h,z_{hh})$ defined in~\eqref{ei:a} and real error
\begin{equation}\label{effind}
  \text{eff}_h := \frac{\eta_h(u_h,z_{hh})}{J(u-u_h)}.
\end{equation}
For linear elliptic problems one usually observes effectiveness going to 1, as $h\to 0$.

The DWR method is easily extended to nonlinear problems, to systems of
differential equations and to time dependent problems. All these and 
further extensions and various applications have already been
demonstrated by Becker and	Rannacher~\cite{BeckerRannacher2001}. 
The fundamental problem that is
still open is a reliable and efficient procedure for approximating
the weights and, in the case of nonlinear problems, bounds on a higher
order remainder that must usually be dropped. 

\subsection{Estimating the network error for the Laplace
  equation}\label{sec:dwr:net} 

Within this framework we now aim at estimating the functional error of the 
neural network solution $u_\N\in V_\N$ obtained with the \emph{Deep Ritz} 
approach. Since the network minimizer $u_\N\not\in\V=H^1_0(\Omega)$ does not 
satisfy the Dirichlet condition $u=0$, the error $\phi:=u-u_\N$ is no 
admissible test function for the variational formulation of the adjoint 
problem~\eqref{adjoint:var}. Hence, we multiply both sides of the classical 
formulation $-\Delta z = J$, see~\eqref{adjoint:strong},  with $\phi=u-u_\N$, 
integrate over $\Omega$ and a consistency term on the boundary remains
\begin{equation}\label{ne:1}
  J(u-u_\N) = (\nabla (u-u_\N),\nabla z) + \langle \partial_n z,u_\N\rangle_{\partial\Omega}.
\end{equation}
With $(\nabla u,\nabla z)=(f,z)$ we  derive the error identity
\begin{equation}\label{ne:2}
  J(u-u_\N) = (f,z)-(\nabla u_\N,\nabla z) + \langle \partial_n z,u_\N\rangle_{\partial\Omega}.
\end{equation}
Again, we must approximate $z\in \V$ by a discrete solution which
is accurate, efficiently achievable and which does not fall into the
vicinity of Galerkin orthogonality, which, in the case of the neural
network error $u-u_\N$ imposes the condition $z_H\not\in
V_\N$. Since for the neural network spaces it naturally holds 
$V_h\not\subset V_\N$. We will approximate the adjoint solution in coarse 
finite element spaces, i.e.  $z_H\in V_h\subset \V$. 
Hereby, we introduce the a posteriori error estimator $\eta(u_\N,z_H)$ as
\begin{equation}\label{ee:net}
  \eta(u_\N,z_H) := (f,z_H) - (\nabla u_\N,\nabla z_H) + \langle\partial_n
  z_H, u_\N\rangle_{\partial\Omega}.
\end{equation}
This error estimator is efficiently evaluated on the finite element mesh using a numerical quadrature rule within the domain and along the boundaries. 
The accuracy of the estimate is measured by means of the
effectivity index~\eqref{effind}. We finally note that the
error estimator~\eqref{ee:net} is not specific to the \emph{Deep Ritz}
method. Instead it could also be used in the context of
\emph{DeepXDE} \cite{Lu2021}, or for any other
approximation technique that yields a $H^1$-conforming
solution.

\begin{remark}[Considering high dimensional problems]\label{remark:high}
The \emph{Deep Ritz} method~\cite{EYu2017} has the potential to be
more efficient than conventional grid-based methods such as the finite
element method, especially for high-dimensional problems. The
application of the error estimator to this case will naturally raise
doubts, since for very high dimension the approximation of the dual
problem would not be feasible. In that case, the dual problem should also be
represented using a neural network. Since it is necessary 
that the reconstruction of the dual solution comes from a space that
is not included in the primal space, a different network architecture
or, for example, a different activation function should be used for
the dual solution.

To be specific, let ${\cal A}$ be a different network architecture and 
let $z_{\cal A}\in V_{\cal A}$ be the adjoint solution herein. Then, the 
estimator can be estimates as 
  \[
  \eta(u_\N,z_{\cal A}) := (f,z_{\cal A}) - (\nabla u_\N,\nabla z_{\cal A})+\langle
  \partial_n z_{\cal A},u_\N\rangle_{\partial\Omega}. 
  \]
  In high dimensions, these integrals cannot be efficiently evaluated by standard mesh-based quadrature rules. Instead, also the estimator must be approximated using stochastic integration via
  \begin{multline*}
    \eta(u_\N,z_{\cal A}) := \frac{1}{N^{in}}\sum_{k=1}^{N^{in}}
    f(x_k^{in})z_{\cal A}(x_k^{in}) - \nabla u_\N(x_k^{in})\cdot \nabla z_{\cal A}(x_k^{in})\\
    +\frac{1}{N^{bnd}}\sum_{j=1}^{N^{bnd}}
    \partial_n z_{\cal A}(x_j^{bnd})u_\N(x_j^{bnd}).
  \end{multline*}

\end{remark}

\subsection{Estimating the network error for the Stokes equations}

The estimate can directly be transferred to the Stokes equations,
where we approximate the solution based on the penalized energy
form~(\ref{energy:stokes}) such as described in
Section~\ref{intro:stokes}. For a linear goal functional
$J:H^1_0(\Omega)^2\to\mathds{R}$ we introduce the adjoint solution
\begin{equation}\label{adjoint:stokes}
  \div\,\zt = 0,\; -\Delta \zt - \nabla q = J\text{ in }\Omega,\quad
  \zt=0\text{ on }\partial\Omega. 
\end{equation}
The error identity for the network solution $\vt_\N$
minimizing~(\ref{energy:stokes}) is then derived as
\[
\begin{aligned}
  J(\vt-\vt_\N) &= \big(\nabla \zt,\nabla (\vt-\vt_\N)\big) +
  \big(q,\div\,(\vt-\vt_\N)\big)
  -\langle \partial_n\zt+q\nt,\vt-\vt_\N\rangle_{\partial\Omega} \\
  &= (\ft,\zt) - \big(\nabla\vt_\N,\nabla\zt\big)
  -\big(\div\,\vt_\N,q\big)
  +\langle\vt_\N, \partial_n\zt+q\nt\rangle_{\partial\Omega}.
\end{aligned}
\]
To evaluate and approximate this error identity we compute a coarse
finite element approximation $(\zt_H,q_H)\in V_h\times L_h$ 
\[
-(\div\,\zt_H,\xi_H)+(\nabla\zt_H,\nabla\phi_H)+(q_H,\div\,\phi_H) =
J(\phi_H)
\quad\forall (\phi_H,\xi_H)\in V_h\times L_h,
\]
and define the Stokes error estimate as
\begin{equation}\label{stokes:estimate}
  \eta(\vt_\N,\zt_H,q_H):=
  (\ft,\zt_H) - \big(\nabla\vt_\N,\nabla\zt_H\big)
  -\big(\div\,\vt_\N,q_H\big)
  +\langle \partial_n \zt_H+q_H\nt,\vt_\N\rangle_{\partial\Omega}.
\end{equation}

\section{Network architecture and training}\label{sec:setup}

Let us recall from the introduction that the network architecture is denoted by $\N$ and a specific neural network function by  $v_\N\in V_\N$. More precisely, we consider fully connected $L$-layer 
neural networks $\N$ with $N_l$ neurons in the $l$-th layer.
We denote the weight matrix and bias vector in $l$-th layer by 
$\mathbf{W}^l\in\mathbb{R}^{N_l\times N_{l-1}}$ and 
$\mathbf{b}^l\in\mathbb{R}^{N_l}$, respectively. An activation function 
$\sigma$  is applied elementwise.

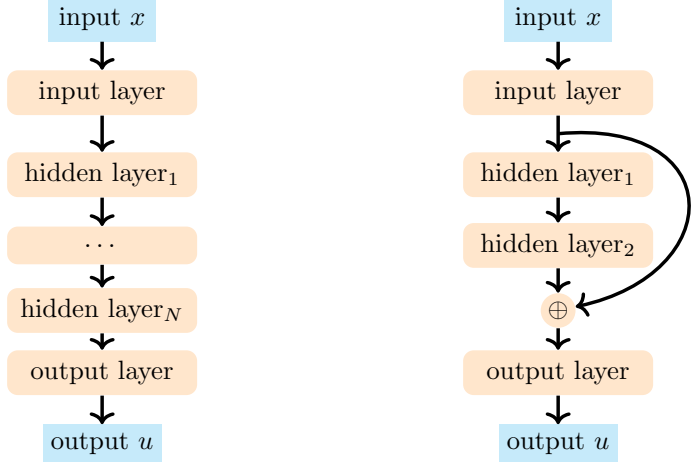
\begin{figure}[t]
  \begin{center}
    \begin{tikzpicture}[scale=0.75]
      \node[rectangle, fill=cyan!20!white,inner sep=4pt] (input) at (1.5, 5.5) 
	   {input $x$};		
	   \node[rectangle, rounded corners, minimum width = 2.5cm, minimum height = 
	     0.5cm, fill=orange!20!white,inner sep=4pt] (layer-begin) at (1.5, 4.2) 
		{input layer};	
		\node[rectangle, rounded corners, minimum width = 2.5cm, minimum height = 
		  0.5cm, fill=orange!20!white,inner sep=4pt] (layer-mid) at (1.5, 2.75) 
		     {hidden layer$_1$};		
		     \node[rectangle, rounded corners, minimum width = 2.5cm, minimum height = 
		       0.5cm, fill=orange!20!white,inner sep=4pt] (layer-end) at (1.5, 1.5) 
			  {$\dots$};		
			  \node[rectangle, rounded corners, minimum width = 2.5cm, minimum height = 
			    0.5cm, fill=orange!20!white,inner sep=4pt] (otimes) at (1.5, .35) {hidden 
			    layer$_N$};		
			  \node[rectangle, rounded corners, minimum width = 2.5cm, minimum height = 
			    0.5cm, fill=orange!20!white,inner sep=4pt] (layer-output) at (1.5, -.75) 
			       {output layer};	
			       \node[rectangle, fill=cyan!20!white,inner sep=3pt] (output) at (1.5, -2.) 
				    {output $u$};
				    
				    \draw[line width=1.3pt,->] (input) -- (layer-begin);
				    \draw[line width=1.3pt,->] (layer-begin) -- (layer-mid);
				    \draw[line width=1.3pt,->] (layer-mid) -- (layer-end);
				    \draw[line width=1.3pt,->] (layer-end) -- (otimes);
				    \draw[line width=1.3pt,->] (otimes) -- (layer-output);
				    \draw[line width=1.3pt,->] (layer-output) -- (output);
				    
				    \node[rectangle, fill=cyan!20!white,inner sep=4pt] (input) at (9.5, 5.5) 
				         {input $x$};		
				         \node[rectangle, rounded corners, minimum width = 2.5cm, minimum height = 
				           0.5cm, fill=orange!20!white,inner sep=4pt] (layer-begin) at (9.5, 4.2) 
				              {input layer};	
				              \node[rectangle, rounded corners, minimum width = 2.5cm, minimum height = 
				                0.5cm, fill=orange!20!white,inner sep=4pt] (layer-mid) at (9.5, 2.75) 
				                   {hidden layer$_1$};		
				                   \node[rectangle, rounded corners, minimum width = 2.5cm, minimum height = 
				                     0.5cm, fill=orange!20!white,inner sep=4pt] (layer-end) at (9.5, 1.5) 
				                        {hidden layer$_2$};		
				                        \node[circle, fill=orange!20!white,inner sep=1pt] (otimes) at (9.5, 0.35) 
				                             {$\oplus$};		
				                             \node[rectangle, rounded corners, minimum width = 2.5cm, minimum height = 
				                               0.5cm, fill=orange!20!white,inner sep=4pt] (layer-output) at (9.5, -.75) 
				                                  {output layer};	
				                                  \node[rectangle, fill=cyan!20!white,inner sep=3pt] (output) at (9.5, -2.) 
				                                       {output $u$};
				                                       
				                                       \draw[line width=1.3pt,->] (input) -- (layer-begin);
				                                       \draw[line width=1.3pt,->] (layer-begin) -- (layer-mid);
				                                       \draw[line width=1.3pt,->] (layer-mid) -- (layer-end);
				                                       \draw[line width=1.3pt,->] (layer-end) -- (otimes);
				                                       \draw[line width=1.3pt,->] (otimes) -- (layer-output);
				                                       \draw[line width=1.3pt,->] (layer-output) -- (output);
				                                       
				                                       \draw[line width=1.3pt,->] 	
				                                       ($(layer-begin)!0.5!(layer-mid)$)..controls(12.5,3.75) and 
				                                       (12.5,1)..(otimes);
				                                       
				                                       
    \end{tikzpicture}
  \end{center}
  \caption{Feed Forward Neural Network and Residual Neural Network}
  \label{fig:network}
\end{figure}

We consider two different architectures. First, a standard \emph{feed
  forward neural network (FFNet)} 
\begin{equation*}
  \begin{aligned}
    v_\N^{(0)}(x) & := =x\in\Omega, \\
    v_\N^{(l)}(x) & :=\sigma(\mathbf{W}^l v_\N^{(l-1)}(x)+\mathbf{b}^l),\quad
    l=1,2,\dots,L,\\
    v_\N(x) &= \mathbf{W}^{L+1}v_\N^{(L)}+\mathbf{b}^{L+1}
  \end{aligned}
\end{equation*}
and second a \emph{residual neural network (ResNet)} that has also
been considered in the original formulation of the \emph{Deep Ritz}
method~\cite{EYu2017} 
\begin{equation*}
  \begin{aligned}
    v_\N^{(0)}&=x,\\
    v_\N^{(l)}(x)&=
    v_\N^{(l-2)}(x)+
    \sigma\left(\mathbf{W}^l\sigma\left(\mathbf{W}^{l-1}
    v_\N^{(l-2)}(x)+\mathbf{b}^{l-1}\right)+\mathbf{b}^l\right),\quad l=2,4,\dots,L,\\
    v_\N(x)&=\mathbf{W}^{L+1}v_\N^{(L)}(x)+\mathbf{b}^{L+1}.
  \end{aligned}
\end{equation*}
With the employed notation $v_\N^{(0)}$ is the input layer with $l_0=d$ and 
$v_\N$ is the output layer with $N_{l+1}=c$. 
All hidden layers are of the same size $H$, i.e. $l=H$ for $1\leq l \leq L$

\subsection{Training}\label{sec:training}

This section presents some insights into the training 
process for the \emph{Deep Ritz} method applied to Laplace
problem on a L-shaped domain. We refer to Section~\ref{sec:num1}
for a  precise definition of the test case.  Here, we study the
effect of the network architecture, i.e. a Feed Forward Neural
Network (FFNet) and a Residual Neural Network (ResNet) on the
training. To train the neural network, we use the Adam optimizer~\cite{kingma2017adam}.

\begin{figure}[t]
  \begin{center}
    \includegraphics[width=1.\textwidth]{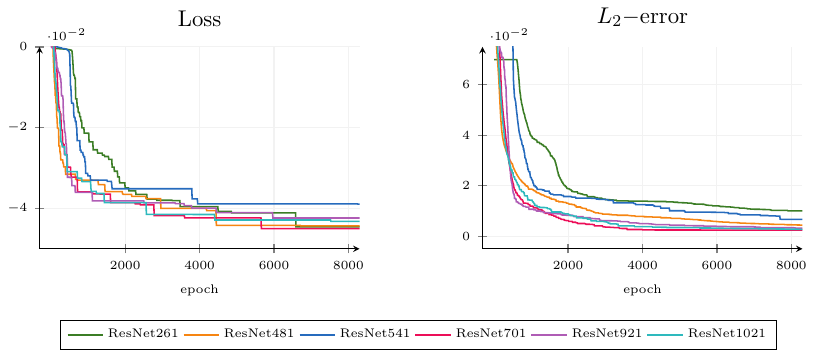}
  \end{center}
  \caption{Training progress for networks of various size.}
  \label{fig:netsize}
\end{figure}

\begin{figure}[t]
  \begin{center}
    \includegraphics[width=1.\textwidth]{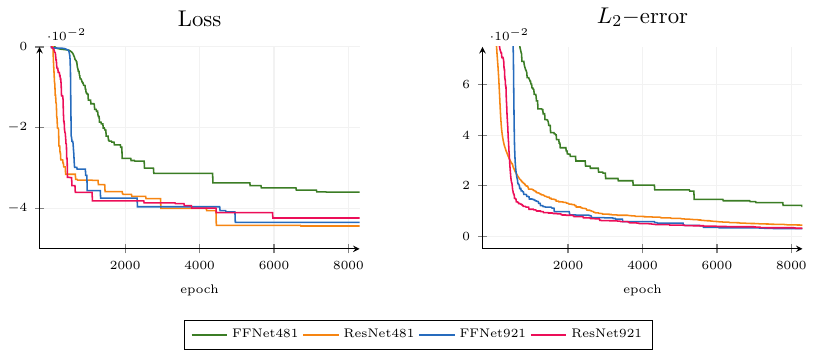}
  \end{center}
  \caption{Training progress for networks of similar size}
  \label{fig:samesize}
\end{figure}

In Figure \ref{fig:netsize} we present training progress for
residual networks of various sizes. In the left
sketch, we show the loss function, i.e. the value of the
penalized and approximated energy
functional~(\ref{lap:lossfunction}) and on the right, we show
the $L^2$ error of the resulting approximations
$\|u-u_{\N}\|$ during training. 
In general, the larger the network, the fewer epochs are needed
to reach a certain error. However, this is not always the
case.  We observe a~certain threshold for the number of
network parameters above which increases the size of the
network does not improve the solution. To further increase the accuracy, 
we would also have to adjust the number of quadrature points accordingly. 
Finally, the network's training plays an important role, making this 
optimisation error difficult to control. In general, the slope
of the loss function is similar to the progress of the
$L^2$-error. Naturally, once low loss levels are reached,
larger networks can yield better approximations. 

In Figure \ref{fig:samesize} we present the training and
approximation progress of networks with  
Feed Forward and Residual architectures and the same sizes. To be
precise, for each architecture, we consider a small network
with $481$ parameters and a larger one with $921$ parameters. The
residual network is faster to train, but this discrepancy gets
smaller for larger networks. The advantage of residual
networks was already mentioned by E and Yu~\cite{EYu2017}. 	

The above considerations show the need for a quality measure
of the solution that works across architectures and training
methods. In the following section, we will present numerical
examples that demonstrate the usability of the error estimator
for controlling the approximation error during training. This
estimate can be used as a stopping criterion once a
sufficiently low error level is reached.

\begin{figure}[t]
  \begin{center}
    \includegraphics[width=0.18\textwidth]{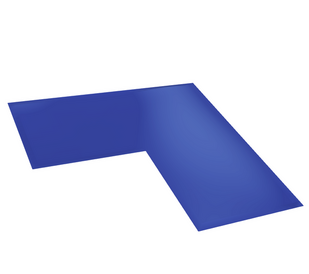}
    \includegraphics[width=0.18\textwidth]{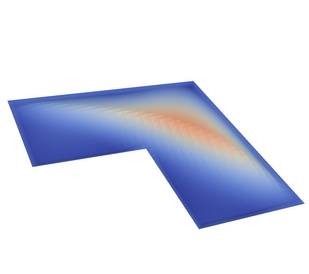}
    \includegraphics[width=0.18\textwidth]{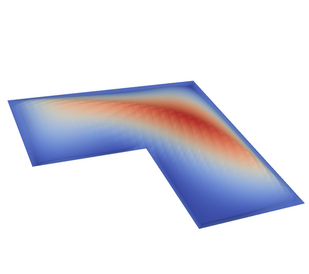}
    \includegraphics[width=0.18\textwidth]{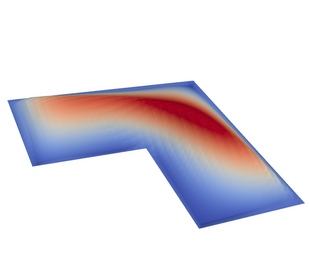}    
    \includegraphics[width=0.18\textwidth]{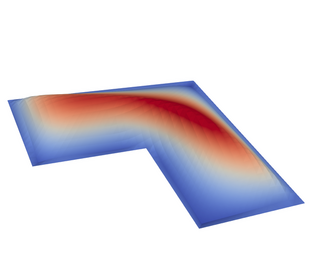}      
    \includegraphics[width=0.18\textwidth]{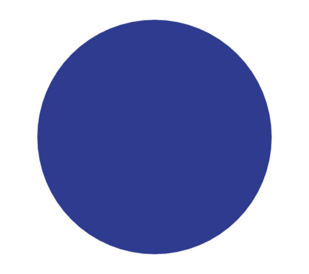}
    \includegraphics[width=0.18\textwidth]{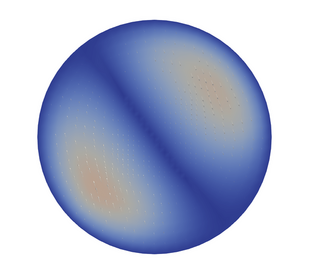}
    \includegraphics[width=0.18\textwidth]{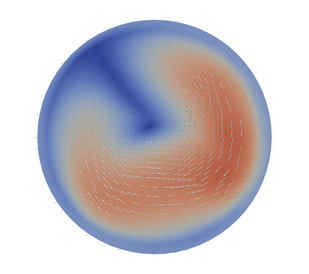}
    \includegraphics[width=0.18\textwidth]{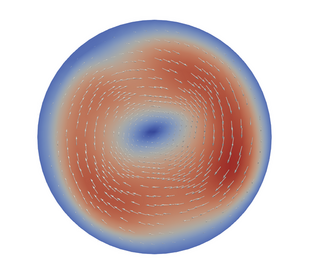}    
    \includegraphics[width=0.18\textwidth]{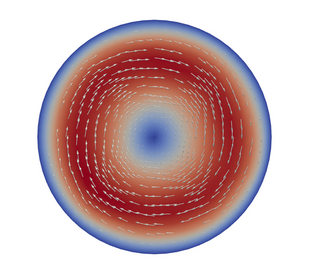}    
    \caption{Solution of Laplace (top) and Stokes (bottom) during training 
      progress and epochs $0,2500,5000,7500,10000$ for the Laplace
      test case and at epochs $0,2000,5000,10000,25000$ in case of the
      Stokes problem.}
    \label{fig:training}
  \end{center}
\end{figure}

\section{Numerical examples}\label{sec:num}

We will discuss two test cases, the Laplace equation on a $L$-shaped
domain and the Stokes equations on a disc.
Figure~\ref{fig:training} shows the solution to both problems obtained
with the \emph{Deep Ritz} method during the network training. 

\subsection{Test Case 1. Laplace equation}\label{sec:num1}

\begin{SCfigure}[2][t]
  \centering
  \includegraphics[width=0.3\textwidth]{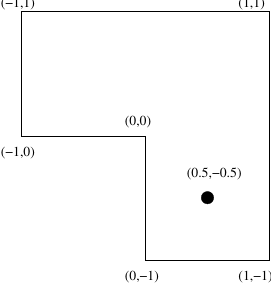}\hspace{0.5cm}
  
  \caption{L-shaped domain and evaluation point $x_a=(0.5,-0.5)$ to
    define test case 1. The Laplace problem is solved with
    homogeneous Dirichlet data and constant right hand side
    $f=1$ such that a corner singularity evolves.}
  \label{fig:ldomain}
\end{SCfigure}

As first test case we consider the Poisson equation on the
$L$-shaped domain $\Omega_L = (-1,1)^2\setminus[0,1]^2$ shown
in Figure~\ref{fig:ldomain}. The quantity of interest is the evaluation of
the solution in the point $x_a=(0.5,-0.5)\in \Omega_L$ 
\begin{equation}\label{ex1:equation}
  -\Delta u = 1\text{ in }\Omega_L,\quad
  u=0\text{ on }\partial\Omega_L,\quad J(u) = u(x_a).
\end{equation}
Since $J\not\in H^{-1}(\Omega_L)$ is not an admissible functional it
should be replaced by averaging over a small neighbourhood of 
the point $x_a$. 
	This is discussed in~\cite[Sec. 5.2]{RichterWick2015}. 
	Also, the reentrant corner of the L-shaped geometry reduces the regularity of the solution, so that the superapproximation results, which are the basis for reconstructing the solution, cannot be used stringently. On the other hand, it is well documented that the 
non-regularized functional limited solution regularity nevertheless gives optimal performance in the context of the dual 
weighted residual method, see~\cite{BeckerRannacher2001,RichterWick2015}. For 
comparison, we first determine a reference value by finite element 
simulations on highly refined  meshes. We identify it as
\[
J_{ref} = 0.1024 \pm 0.0020.
\]
First we demonstrate the performance of the DWR estimator 
\[
\eta(u_{\cal N},z_h) = (f,z_h) - (\nabla u_\N,\nabla z_h) + 
\langle \partial_n z_h, u_\N\rangle_{\partial\Omega}
\]
as presented in Section~\ref{sec:dwr:net}. 
The adjoint solution $z_h$ will be computed as finite element
approximation on very coarse meshes. 
Training results and estimator values are shown for neural
network solutions obtained with the \emph{Deep Ritz} method
and using the strong formulation. Both network architectures
of FFNet type and of ResNet type are considered. The complete
set of parameters is summarized as follows:
\begin{itemize}
\item FFNet:\, $H = 20$, $L=4$, $\sigma(x) = \operatorname{ELU}(x)$,
\item ResNet: $H = 20$, $L=2$, $\sigma(x) = \max(x^3,0)$.
\end{itemize}
Since each ResNet block consists of two layers, both architectures
have same number of parameters. The \emph{Exponential Linear Unit
  (ELU)} is  defined as  
\[
\operatorname{ELU}(x)=\left\{\begin{array}{ll}
x & \text { if } x\ge 0 \\
e^{x}-1 & \text { if } x<0.
\end{array}\right.
\]

\begin{figure}[t]
  \begin{center}
    \includegraphics[width=1.\textwidth]{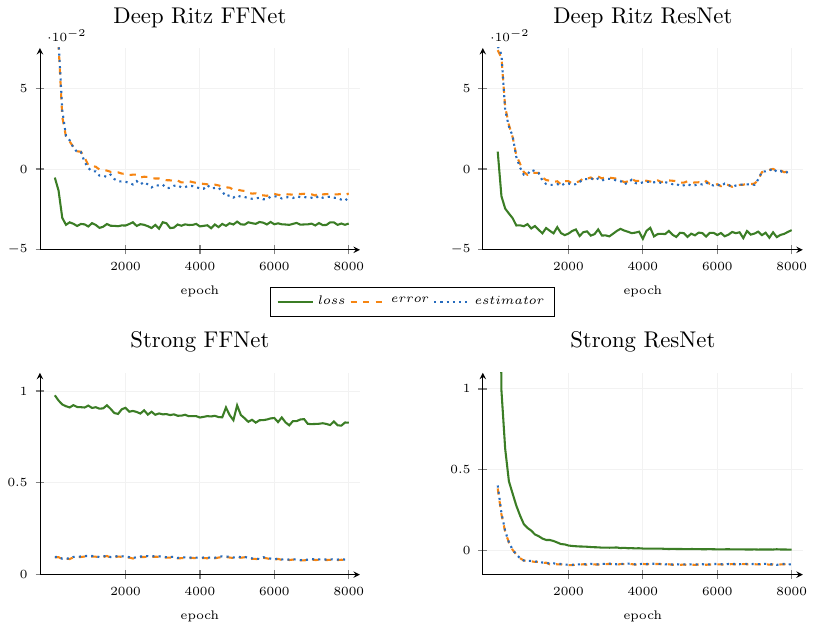}
  \end{center}
  \caption{Loss, error and estimator for different network architectures and 
    Loss functions.}
  \label{fig:dr_strong_training}
\end{figure}

For the \emph{Deep Ritz} approach and the strong formulation
all gradients are computed both with automatic differentiation
and finite difference approximation, respectively. We perform
8000 epochs and  the estimator is evaluated every $100$ epochs,
see~Figure~\ref{fig:dr_strong_training}, where we show the loss
function (bold green line), the functional error $J(u-u_\N)$ (dashed
orange line) and the error estimator $\eta(u_\N,z_h)$ (dotted blue
line). 
We note that functional errors $J(u)-J(u_\N)$ are
generally signed. Hence, convergence in
Fig.~\ref{fig:dr_strong_training} cannot be expected to be monotone
and also, errors are not necessarily positive.

One can observe that independently of the applied method, the
estimator follows the error and gives a highly accurate error approximation. 
Consequently we study the dependence on the coarse mesh size
$h\in \{0.5,0.25, 0.0625 \}$ used to approximate the adjoint
solution and show the results in
Figure~\ref{fig:levels}. Increasing the level of  
refinement improves the exactness of the estimator. 
The results are good even for extremely coarse meshes. 
The estimator is highly efficient and cheap to evaluate such that it brings along very little computational overhead. This allows to use the estimator as stopping criterium while training the network.
The choice
$h=0.5$ corresponds to only $12$ quadrilateral elements, the
finest mesh with $h=0.0625$ corresponds to just $768$
elements. Further, many degrees of freedom reside on the
boundary of the domain such that the number of unknowns to
approximate the adjoint solution ranges from $5$ on the coarsest
mesh to $640$ on the finest mesh and from the results we observe
that the intermediate mesh with $h=0.25$ comprising $16$
unknowns is sufficiently accurate.

\begin{figure}[t]
  \begin{center}
    \includegraphics[width=1.\textwidth]{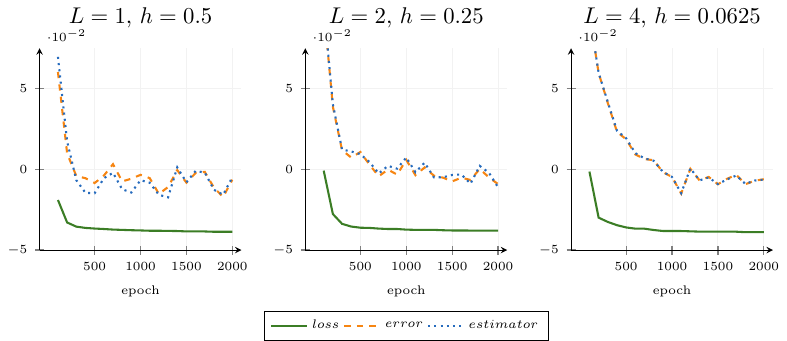}
  \end{center}
  \caption{Loss, error and estimator for different refinement levels of dual 
    solution \mbox{$h=2^{-L}$.}}
  \label{fig:levels}
\end{figure}

\begin{table}[t]
  \centering
  \resizebox{\textwidth}{!}{
    \begin{tabular}{r|ccc|ccc|ccc}
      \toprule
      & \multicolumn{3}{c}{L=1}  & \multicolumn{3}{c}{L=2} & 
      \multicolumn{3}{c}{L=4} \\  
      \text{epoch} & \text{error} & \text{estimate} & $\text{eff}_h$ & 
      \text{error} & \text{estimate} & $\text{eff}_h$& \text{error} & 
      \text{estimate} & $\text{eff}_h$
      \\	\midrule
      500 & -0.008480 & -0.014786 & 0.57 & 
      0.010748 &	0.009397 & 1.14 & 
      0.018578 &	0.019248 & 0.96 \\ 
      1000 & -0.003403 & -0.006700 & 0.51 &
      0.005855 &	0.007324 & 0.80 &
      -0.004695 & -0.004719 & 0.99 \\
      1500 & -0.008228 & -0.008022 & 1.03 &
      -0.007460 & -0.003477 & 2.15 &
      -0.009242 & -0.009126 & 1.01 \\
      2000 & -0.006424 & -0.004946 & 1.23 &
      -0.009138 & -0.011058 & 0.83 &
      -0.006117 & -0.006393 & 0.96 \\
      \bottomrule
  \end{tabular}}
  \caption{The values of error and estimator with
    effectivity index $\text{eff}_h$, see (\ref{effind}), for
    different refinement levels of dual solution
    \mbox{$h=2^{-L}$.}} 
  \label{tab:lapalce}
\end{table}

\subsection{Test Case 2. Stokes equations}

For the second test case we consider the Stokes equations on the
unit circle $\Omega = \{x\in\mathds{R}^2,\; |x|_2<1\}$. We prescribe an
analytical solution for comparison with the neural network  
approximation given by 
\[
\vt(x,y) =
\cos\left(\frac{\pi}{2}(x^2+y^2)\right)\begin{pmatrix}\phantom{-}y \\ -x
\end{pmatrix}
\]
and compute the corresponding forcing term as 
\[
\ft(x,y) =\pi \cos\left(\frac{\pi}{2}(x^2+y^2)\right)
\begin{pmatrix} \phantom{-}y(x^2+y^2) \pi 
  + 
  4(y-x)\tan\left(\frac{\pi}{2}(x^2+y^2)\right)\\
  -x(x^2+y^2) \pi
  -  4(x+y)\tan\left(\frac{\pi}{2}(x^2+y^2)\right) 
\end{pmatrix}.
\]
The functional of interest $J(\vt)$ is an integral of a
y-component of the velocity on a line segment $[0,1]$
\[
J(\vt):= \int_0^1 \vt_y(x,0) \dd x ,\quad J_{ref} = -\frac{1}{\pi}.
\]
In Figure~\ref{fig:eststokes} we present the loss function and the 
functional error as well as the error estimator. The training of the
Deep Ritz method is performed for  $25000$ epochs, with the
Feedforward Neural Network (FFNet: $d = 2$, $c=2$, $H = 10$, $L=20$,
$\sigma(x) = \operatorname{ELU}(x)$).  
The adjoint Stokes problem is approximated with equal order finite
elements using pressure stabilization on a coarse mesh level $L=3$
that corresponds to   $h\approx 0.04375$. The results for some
selected  epochs together with effectivity  index are summarized in
Table~\ref{tab:stokes}. The error estimator is highly accurate and robust over the complete training process such that it can be used as stopping criterium.
In particular for $L=4$ the effectivities are very close to one and the 
estimator values deviate by less than $5\%$ from the true error. Given the 
limited regularity of the problem (reentrant corner and quantity of interest 
that is not a linear functional in $H^1(\Omega)$) this result is remarkable. 
  
\begin{figure}[t]
\begin{center}
    \includegraphics[width=1.\textwidth]{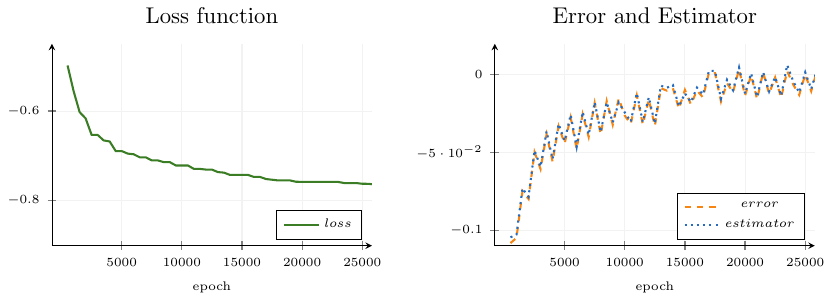}
  \end{center}
  \caption{Loss, error and estimator for Stokes problem.}
  \label{fig:eststokes}
\end{figure}

\begin{table}[t]
  \centering
  \begin{tabular}{r|ccc}
    \toprule
    \text{epoch} & \text{error} & \text{estimate} & $\text{eff}_h$\\	\midrule
    5000 & -0.0442546 & -0.0425343 & 1.04 \\ 
    10000 & -0.0262763 & -0.0241991 & 1.08 \\
    15000 & -0.0096646 & -0.0118802 & 0.81 \\
    20000 & -0.0134781 & -0.0121668 & 1.10 \\
    \bottomrule
  \end{tabular}
  \caption{The values of error and estimator with
    effectivity index $\text{eff}_h$, see \eqref{effind}.}
  \label{tab:stokes}
\end{table}

The numerical implementation is realized in the finite element toolkit
Gascoigne 3D~\cite{Gascoigne3d}, which is coupled to the machine learning framework PyTorch \cite{PyTorch}.

\section{Conclusion}\label{sec:conc}

In this article, we have used different tools from the finite element
analysis to get a deeper insight into the neural network approximation of
partial differential equations obtained with \emph{Physics-Informed
Neural Networks}. In particular, we used standard tools of error analysis 
to interpret the generalization error as consistency error arising from faulty 
numerical quadrature. Further, based on the \emph{dual weighted
residual method} we have derived an a posteriori error estimator that can be 
used to measure the error of previously defined networks.
The efficiency and accuracy of the estimator has been numerically demonstrated 
in applications to the Laplace and the Stokes problem. 
The method is independent of  
the design of the neural network and the training procedure. The
evaluation on a very coarse meshes already shows very good accuracy,
such that little computational overhead is brought along. The
estimator can be used as an accurate and straightforward stopping
criterion during the training process. Hereby, we gain a first
validation of the neural network approximation, and the error
controlled training also helps reduce the computational effort by
avoiding excessive training epochs.

\section*{Acknowledgements}
Both authors acknowledge the financial support by the Deutsche 
Forschungsgemeinschaft, GRK 2297 MathCoRe, grant number 314838170. Furthermore, 
we thank the anonymous reviewers for their comments that helped us to improve 
the manuscript.


\begin{thebibliography}{10}
  \expandafter\ifx\csname url\endcsname\relax
  \def\url#1{\texttt{#1}}\fi
  \expandafter\ifx\csname urlprefix\endcsname\relax\def\urlprefix{URL }\fi
  \expandafter\ifx\csname href\endcsname\relax
  \def\href#1#2{#2} \def\path#1{#1}\fi
  
  \bibitem{EYu2017}
  W.~E, B.~Yu, The {D}eep {R}itz method: A deep learning-based numerical
  algorithm for solving variational problems, Communications in Mathematics and
  Statistics 6~(1) (2018) 1--12.
  \newblock \href {https://doi.org/10.1007/s40304-018-0127-z}
  {\path{doi:10.1007/s40304-018-0127-z}}.
  
  \bibitem{Samaniego2020}
  E.~Samaniego, C.~Anitescu, S.~Goswami, V.~Nguyen-Thanh, H.~Guo, K.~Hamdia,
  X.~Zhuang, T.~Rabczuk, An energy approach to the solution of partial
  differential equations in computational mechanics via machine learning:
  Concepts, implementation and applications, Computer Methods in Applied
  Mechanics and Engineering 362 (2020) 112790.
  \newblock \href {https://doi.org/10.1016/j.cma.2019.112790}
  {\path{doi:10.1016/j.cma.2019.112790}}.
  
  \bibitem{Lu2021}
  L.~Lu, X.~Meng, Z.~Mao, G.~E. Karniadakis, {DeepXDE}: A deep learning library
  for solving differential equations, {SIAM} Review 63~(1) (2021) 208--228.
  \newblock \href {https://doi.org/10.1137/19m1274067}
  {\path{doi:10.1137/19m1274067}}.
  
  \bibitem{Berg2018}
  J.~Berg, K.~Nyström, A unified deep artificial neural network approach to
  partial differential equations in complex geometries, Neurocomputing 317
  (2018) 28--41.
  \newblock \href {https://doi.org/10.1016/j.neucom.2018.06.056}
  {\path{doi:10.1016/j.neucom.2018.06.056}}.
  
  \bibitem{Sirignano2018}
  J.~Sirignano, K.~Spiliopoulos, {DGM}: A deep learning algorithm for solving
  partial differential equations, Journal of Computational Physics 375 (2018)
  1339--1364.
  \newblock \href {https://doi.org/10.1016/j.jcp.2018.08.029}
  {\path{doi:10.1016/j.jcp.2018.08.029}}.
  
  \bibitem{Anitescu2019}
  C.~Anitescu, E.~Atroshchenko, N.~Alajlan, T.~Rabczuk, Artificial neural 
  network
  methods for the solution of second order boundary value problems, Computers,
  Materials \& Continua 59~(1) (2019) 345--359.
  \newblock \href {https://doi.org/10.32604/cmc.2019.06641}
  {\path{doi:10.32604/cmc.2019.06641}}.
  
  \bibitem{KharazmiZhangKarniadakis2019}
  E.~Kharazmi, Z.~Zhang, G.~E. Karniadakis, Variational physics-informed neural
  networks for solving partial differential equations, arXiv 1912.00873 (2019).
  
  \bibitem{Pinkus1999}
  A.~Pinkus, Approximation theory of the {MLP} model in neural networks, Acta
  Numerica 8 (1999) 143–195.
  \newblock \href {https://doi.org/10.1017/S0962492900002919}
  {\path{doi:10.1017/S0962492900002919}}.
  
  \bibitem{Ghring2019}
  I.~G\"{u}hring, G.~Kutyniok, P.~Petersen, Error bounds for approximations with
  deep {ReLU} neural networks in ${W}^{s,p}$ norms, Analysis and Applications
  (2019) 1--57\href {https://doi.org/10.1142/s0219530519410021}
  {\path{doi:10.1142/s0219530519410021}}.
  
  \bibitem{deeponet}
  L.~Lu, P.~Jin, G.~Pang, Z.~Zhang, G.~Karniadakis, Learning nonlinear operators
  via {DeepONet} based on the universal approximation theorem of operators, Nat
  Mach Intell 3 (2021) 218--229.
  \newblock \href {https://doi.org/10.1038/s42256-021-00302-5}
  {\path{doi:10.1038/s42256-021-00302-5}}.
  
  \bibitem{Mishra2018}
  S.~Mishra, A machine learning framework for data driven acceleration of
  computations of di erential equations, Mathematics in Engineering 1~(1)
  (2018) 118--146.
  \newblock \href {https://doi.org/10.3934/mine.2018.1.118}
  {\path{doi:10.3934/mine.2018.1.118}}.
  
  \bibitem{Brevis2020}
  I.~Brevis, I.~Muga, K.~G. van~der Zee, A machine-learning minimal-residual
  ({ML}-{MRes}) framework for goal-oriented finite element discretizations,
  Computers {\&} Mathematics with Applications (Sep. 2020).
  \newblock \href {https://doi.org/10.1016/j.camwa.2020.08.012}
  {\path{doi:10.1016/j.camwa.2020.08.012}}.
  
  \bibitem{HartmannLessigMargenbergRichter2020}
  N.~Margenberg, D.~Hartmann, C.~Lessig, T.~Richter, A neural network multigrid
  solver for the {N}avier-{S}tokes equations, Journal of Computational Physics
  (2022) 110983\href {http://arxiv.org/abs/2008.11520}
  {\path{arXiv:2008.11520}}, \href {https://doi.org/10.1016/j.jcp.2022.110983}
  {\path{doi:10.1016/j.jcp.2022.110983}}.
  
  \bibitem{Margenberg_2021}
  N.~Margenberg, C.~Lessig, T.~Richter, Structure preservation for the deep
  neural network multigrid solver, {ETNA} - Electronic Transactions on
  Numerical Analysis 56 (2021) 86--101.
  \newblock \href {https://doi.org/10.1553/etna_vol56s86}
  {\path{doi:10.1553/etna_vol56s86}}.
  
  \bibitem{RothSchroederWick2022}
  J.~Roth, M.~Schröder, T.~Wick, Neural network guided adjoint computations in
  dual weighted residual error estimation, SN Applied Sciences (accepted2022).
  
  \bibitem{minakowski2021error}
  P.~Minakowski, T.~Richter, Error estimates for neural network solutions of
  partial differential equations (2021).
  \newblock \href {http://arxiv.org/abs/2107.11035} {\path{arXiv:2107.11035}}.
  
  \bibitem{BeckerRannacher2001}
  R.~Becker, R.~Rannacher, An optimal control approach to a posteriori error
  estimation in finite element methods, Acta Numerica (2001) 1--225.
  
  \bibitem{MuellerZeinhofer2021}
  J.~M\"uller, M.~Zeinhofer, Error estimates for the variational training of
  neural networks with boundary penalty (2021).
  \newblock \href {http://arxiv.org/abs/2103.01007} {\path{arXiv:2103.01007}}.
  
  \bibitem{Liao2021}
  Y.~Liao, P.~Ming, Deep {N}itsche method: Deep {R}itz method with essential
  boundary conditions, Commun. Comput. Phys. 29~(5) (2021) 1365--1384.
  
  \bibitem{Dondl2021}
  P.~Dondl, J.~Müller, M.~Zeinhofer, Uniform convergence guarantees for the
  {D}eep {R}itz method for nonlinear problems (2021).
  \newblock \href {http://arxiv.org/abs/2111.05637} {\path{arXiv:2111.05637}}.
  
  \bibitem{LuLuWang2021}
  J.~Lu, Y.~Lu, M.~Wang, A priori generalization analysis of the {D}eep {R}itz
  method for solving high dimensional elliptic partial differential equations,
  in: 34th Annual Conference on Learning Theory, Vol. 134, 2021, pp. 1--64.
  
  \bibitem{Barron1993a}
  A.~Barron, {Universal Approximation Bounds for Superpositions of a Sigmoidal
    Function}, IEEE Transactions on Information Theory 39~(3) (1993) 930--945.
  \newblock \href {https://doi.org/10.1109/18.256500}
  {\path{doi:10.1109/18.256500}}.
  
  \bibitem{Mishra2022}
  S.~Mishra, R.~Molinaro, Estimates on the generalization error of physics
  informed neural networks ({PINNs}) for approximating {PDEs}, IMA Journal of
  Numerical Analysis (2022).
  \newblock \href {https://doi.org/10.1093/imanum/drab093}
  {\path{doi:10.1093/imanum/drab093}}.
  
  \bibitem{HongSiegelXu2021}
  Q.~Hong, J.~Siegel, J.~Xu, A priori analysis of stable neural network 
  solutions
  to numerical {PDEs} (2021).
  \newblock \href {http://arxiv.org/abs/2104.02903v3}
  {\path{arXiv:2104.02903v3}}.
  
  \bibitem{Caflisch1998}
  R.~Caflisch, Monte {C}arlo and quasi-{M}onte {C}arlo methods, Acta Numerica
  (1998) 1--49.
  
  \bibitem{John2016}
  V.~John, Finite Element Methods for Incompressible Flow Problems, Vol.~51 of
  Computational Mathematics, Springer, 2016.
  
  \bibitem{BeckerRannacher1995}
  R.~Becker, R.~Rannacher, Weighted a posteriori error control in {FE} methods,
  in: e.~a. H.~G.~Bock (Ed.), ENUMATH'97, World Sci. Publ., Singapore, 1995.
  
  \bibitem{BeckerVexler2004}
  R.~Becker, B.~Vexler, A posteriori error estimation for finite element
  discretization of parameter identification problems, Numer. Math. 96~(3)
  (2004) 435--459.
  
  \bibitem{RannacherSuttmeier1998}
  R.~Rannacher, F.-T. Suttmeier, A posteriori error control in finite element
  methods via duality techniques: Application to perfect plasticity,
  Computational Mechanics 21 (1998) 123--133.
  
  \bibitem{Richter2017}
  T.~Richter, Fluid-structure Interactions. Models, Analysis and Finite 
  Elements,
  Vol. 118 of Lecture notes in computational science and engineering, Springer,
  2017.
  
  \bibitem{BesierRannacher2012}
  M.~Besier, R.~Rannacher, Goal-oriented space-time adaptivity in the finite
  element galerkin method for the computation of nonstationary incompressible
  flow 70~(9) (2012) 1139--1166.
  
  \bibitem{MeidnerRichter2015}
  D.~Meidner, T.~Richter, A posteriori error estimation for the fractional step
  theta discretization of the incompressible {N}avier-{S}tokes equations, Comp.
  Meth. Appl. Mech. Engrg. 288 (2015) 45--59.
  
  \bibitem{RichterWick2015}
  T.~Richter, T.~Wick, Variational localizations of the dual weighted residual
  estimator, Journal of Computational and Applied Mathematics 279 (2015)
  192--208.
  \newblock \href {https://doi.org/10.1016/j.cam.2014.11.008}
  {\path{doi:10.1016/j.cam.2014.11.008}}.
  
  \bibitem{kingma2017adam}
  D.~P. Kingma, J.~Ba, Adam: A method for stochastic optimization (2017).
  \newblock \href {http://arxiv.org/abs/1412.6980} {\path{arXiv:1412.6980}}.
  
  \bibitem{Gascoigne3d}
  M.~Braack, R.~Becker, D.~Meidner, T.~Richter, B.~Vexler, The finite element
  toolkit gascoigne, \url{www.gascoigne.de} (Oct. 2021).
  \newblock \href {https://doi.org/10.5281/zenodo.5574969}
  {\path{doi:10.5281/zenodo.5574969}}.
  
  \bibitem{PyTorch}
  A.~Paszke, S.~Gross, F.~Massa, A.~Lerer, J.~Bradbury, G.~Chanan, T.~Killeen,
  Z.~Lin, N.~Gimelshein, L.~Antiga, A.~Desmaison, A.~Kopf, E.~Yang, Z.~DeVito,
  M.~Raison, A.~Tejani, S.~Chilamkurthy, B.~Steiner, L.~Fang, J.~Bai,
  S.~Chintala, Pytorch: An imperative style, high-performance deep learning
  library, in: H.~Wallach, H.~Larochelle, A.~Beygelzimer, F.~d\textquotesingle
  Alch\'{e}-Buc, E.~Fox, R.~Garnett (Eds.), Advances in Neural Information
  Processing Systems 32, Curran Associates, Inc., 2019, pp. 8024--8035.
  
\end{thebibliography}

\end{document}